\documentclass[11pt]{amsart}
\usepackage{amsmath}
\usepackage{amssymb}
\usepackage{mathtools}
\usepackage[usenames,dvipsnames]{xcolor}
\usepackage[margin=1in]{geometry}
\usepackage{hyperref}

\def\Z{{\mathbb Z}}
\def\N{{\mathbb N}}
\def\Q{{\mathbb Q}}

\DeclareMathOperator{\Kn}{Kn}

\def\<{\langle}
\def\>{\rangle}
\newtheorem{thm}{Theorem}[section]
\newtheorem{lem}[thm]{Lemma}
\newtheorem{cor}[thm]{Corollary}
\newtheorem{prop}[thm]{Proposition}

\newtheorem{quest}[thm]{Question}

\theoremstyle{definition}
\newtheorem{defn}[thm]{Definition}
\theoremstyle{remark}
\newtheorem{rem}[thm]{Remark}
\theoremstyle{remark}
\newtheorem{ex}[thm]{Example}
\theoremstyle{remark}

\begin{document}

\title{Strongly regular graphs, strongly polynomial sequences, and two-level polynomials}
\author{Tristram Bogart and Santiago Jim\'enez-Salazar}

\begin{abstract}
We study the behavior of the characteristic and chromatic polynomials of certain families of graphs and their relation to homomorphism counts, focusing on polynomial dependence of the coefficients at each fixed codegree on the graph family parameter. We prove that the
characteristic polynomials of strongly regular graphs with polynomial parameters form a two-level polynomial of infinite depth in the sense of Bogart and Woods. In contrast, the chromatic polynomials of Paley
graphs have maximal constant depth two: their coefficient of codegree three is not
eventually polynomial in the number of vertices. We also prove that both the
chromatic and characteristic polynomials of every strongly polynomial
graph sequence, in the sense of de La Harpe and Jaeger, have infinite depth. Finally, we investigate homomorphism counts between two strongly polynomial graph sequences.
We show that infinite depth need not hold in general, but we construct target
sequences for which it holds for every strongly polynomial source sequence.
 
\end{abstract}

\maketitle

\section{Introduction} 

For a graph polynomial $f_G(n)$ and a family of graphs $\{G_q\}$ indexed by a natural number parameter $q$, we consider the sequence of polynomials $f_{G_q}(n)$ as a function of two variables $q$ and $n$.  
Although the degree of $f_{G_q}$ may grow with $q$, its
coefficients often exhibit polynomial dependence on $q$ when indexed by
their codegree (that is, their distance from the leading term.) The notion of a \emph{two-level polynomial}, inspired by Chaiken, Hanusa, and Zaslavsky's study of nonattacking chess queens \cite{CHZ} and formalized and generalized by Bogart and Woods~\cite{BW}, provides a framework for studying this behavior. Among other combinatorial applications, the latter paper shows that for certain families of graphs, both the characteristic and chromatic polynomials exhibit two-level behavior. 

In this paper, we systematically study two-level behavior for chromatic and characteristic
polynomials of more general families of graphs. We first consider \emph{strongly regular graphs}, which are regular graphs satisfying the (highly restrictive) additional conditions that every pair of adjacent vertices have $\lambda$ common neighbors and every pair of non-adjacent vertices have $\mu$ common neighbors for some $\lambda, \mu$. We show that for polynomially parametrized families of strongly regular graphs $G_q$, their characteristic polynomials always exhibit strong two-level behavior. In contrast, we use the strongly regular family of \emph{Paley graphs} to show that the chromatic polynomials may strikingly fail to do so. 

We also analyze the close relationship between two-level behavior and \emph{strongly polynomial sequences} of graphs $\{G_q\}$. A sequence $\{G_q\}$ is strongly polynomial if, for every fixed graph $F$, the number $\hom(F,G_q)$ of homomorphisms from $F$ to $G_q$ is given by a polynomial in $q$. The notion originates in the work of de la Harpe and Jaeger~\cite{Jaeger} on generalizations of the chromatic polynomial by homomorphism counting. Goodall, Ne\v{s}et\v{r}il, and Ossona de Mendez~\cite{GNO} extend this viewpoint to relational structures and use quantifier-free interpretations to construct a broad class of strongly polynomial graph sequences. 

We begin by reviewing the key definitions and properties of two-level polynomials from \cite{BW}, along with some examples and a new construction. 

\begin{defn} A \emph{numerical polynomial} is a polynomial $g(x) \in \Q[x]$ such that $g(q) \in \Z$ for every $q \in \Z$.
\end{defn}

\begin{defn} \cite{BW}
\label{def:two-level}
Let $g$ be an eventually nonnegative numerical polynomial, let $S=\{q\in\N:\ g(q)\ge 0\}$, and let $e:S \to \N \cup \{-1,\infty\}$.
\begin{enumerate}
\item[(a)] A \emph{two-level polynomial} of degree function $g=g(q)$ and depth function $e=e(q)$ is an infinite sequence of polynomials $\{f_q\}_{q\in S}$, with $f_q(n)$  of degree $g(q)$, that has the following property: writing \[f_q(n) = \sum_{r=0}^{\infty}c_r(q)n^{g(q)-r},\] where $c_r(q)=0$ if $r>g(q)$, there exist polynomials $\{\varphi_r(q)\}_{r\ge 0}$ such that, for every $q\in S$ and every $r \leq e(q)$, we have $c_r(q) = \varphi_r(q)$. 
 \item[(b)] For $q\in \N\setminus S$, define $f_q(n)=0$ and $e(q)=-1$, extending the definition of $f_q(n)$ to all $q\in\N$.
 \end{enumerate}
\end{defn}

\begin{ex}
The sequence $f_q(n)=(n+1)^q$ is a two-level polynomial of infinite depth, because \[f_q(n)=\sum\limits_{r=0}^{\infty}\binom{q}{r}n^{q-r}\] and the coefficient $\binom{q}{r}$ agrees with the polynomial $\frac{q(q-1)\cdots (q-r+1)}{r!}$ even when $r>q$.
\end{ex}

\begin{ex} \label{ex:finitedepth}
For \(q\geq 1\), the family 
$f_q(n)=n^q+n$ has only finite depth $e(q)=q-2$.
Its degree is \(g(q)=q\), and its coefficients, indexed by
codegree, are
\[
c_r(q)=
\begin{cases}
1, & r=0 \text{ or } r=q-1,\\
0, & \text{otherwise}.
\end{cases}
\]
\end{ex}

\begin{rem}
By definition, \emph{any} sequence of polynomials of $q$ whose degree function is polynomial in $q$ is a two-level polynomial of depth -1, and many sequences have constant positive depth for "trivial" reasons. For this reason, we say that a sequence exhibits two-level behavior if the depth function is unbounded. The strongest two-level behavior occurs when the depth is infinite.
\end{rem}

To give an idea of the range of this concept, the following proposition provides an easy method to construct many two-level polynomials with infinite depth. In particular, it shows that the falling factorial $(n)_q$, defined by $(n)_q=n(n-1)\cdots (n-q+1)$, is a two-level polynomial with infinite depth.

\begin{prop}
\label{prop:roots polynomial}
Let $P(q)$ be a polynomial. Then

$$f_q(n)=\prod\limits_{i=0}^{q-1} (n+P(i))$$

is a two-level polynomial of degree $q$ and infinite depth.
\end{prop}

\begin{lem}
Let $p\in \Q[x_1,\ldots,x_r]$. Then
\[
S_p(q):=
\sum_{0\leq i_1<\cdots<i_r<q}
p(i_1,\ldots,i_r)
\]
is a polynomial in $q$.
\end{lem}

\begin{proof}
We proceed by induction on $r$. The case $r=0$ is trivial. For $r\geq 1$, observe that
\[
S_p(q+1)-S_p(q)
=
\sum_{0\leq i_1<\cdots<i_{r-1}<q}
p(i_1,\ldots,i_{r-1},q).
\]
By the induction hypothesis, the finite difference $
S_p(q+1)-S_p(q)
$
is a polynomial in $q$. Therefore, $S_p(q)$ is a polynomial in $q$.
\end{proof}

\begin{proof}[Proof of Theorem \ref{prop:roots polynomial}]
Writing
\[
F_q(n)=\sum_{r=0}^{q}c_r(q)n^{q-r},
\]
we have
\[
c_r(q)
=
\sum_{0\leq i_1<\cdots<i_r<q}
P(i_1)\cdots P(i_r)
=\sum_{0\leq i_1<\cdots<i_r<q}p(i_1,...i_r)\]

where $
p(x_1,\ldots,x_r)
=
P(x_1)\cdots P(x_r)$. Hence, by the lemma, $c_r(q)$ is a polynomial for every $q$, so $f_q(n)$ is a two-level polynomial of degree function $q$ and infinite depth.
\end{proof}

We summarize the key closure properties of the class of two-level polynomials that we will need. These properties were originally stated and proved for the more general class of two-level \emph{quasi}-polynomials, but none of our results will require this generality.

\begin{prop} \cite[\S 2]{BW} \label{prop:closure}
  \begin{enumerate}
  \item Suppose $f(q;n)$ is a bivariate polynomial in $q$ and $n$ whose degree in $n$ is $d$. Then $f_q(n) = f(q;n)$ is a two-level polynomial of constant degree function $g(q) = d$ and infinite depth.
  \item  Let $f_q(n)$ and $f'_q(n)$ be two-level polynomials
    of respective degree functions $g(q)$ and $g'(q)$ and depth functions $e(q)$ and $e'(q)$. Then the product $f_q(n)\cdot f'_q(n)$ is a two-level polynomial of degree function $g(q) + g'(q)$ and depth function $\min(e(q),e'(q))$.
  \item Let $h(q)$ be a numerical polynomial and $f_q(n)$ be a two-level polynomial of degree function $g(q)$ and depth function $e(q)$. Suppose the leading coefficient of $f_q(n)$ is a constant $c$. Then  $\left\{f_q(n)^{h(q)}/c^{h(q)}\right\}_{g(q),h(q)\ge 0}$ is a monic two-level quasi-polynomial of degree function $g(q)h(q)$ and depth function $e(q)$. 
  \end{enumerate}
\end{prop}

To analyze Paley graphs, which are indexed by primes congruent to $1$ modulo
$4$, it will be necessary to consider families of polynomials which are not defined for all natural numbers. For this purpose, we introduce the following definition. 

\begin{defn}
Let \(Q\subseteq\mathbb{N}\) be an infinite set. A family
\((f_q(n))_{q\in Q}\) is a \emph{two-level polynomial over \(Q\)}
if it satisfies Definition 1.2 with the index \(q\)
restricted to \(Q\). In particular, the coefficient identities
\(c_r(q)=\varphi_r(q)\) are required only for \(q\in Q\) and
\(0\leq r\leq e(q)\). The degree function \(g\in\mathbb{Q}[q]\)
is required to take nonnegative integer values on \(Q\).
\end{defn}

Bogart and Woods developed a schema for establishing two-level
behavior in the more general setting of combinatorial counting functions
and two-level quasi-polynomials. Roughly, the schema organizes the
contributions to a fixed codegree into finitely many combinatorial types
of bounded size, whose occurrences can be counted polynomially in the
family parameter. Their applications include graph polynomials and
Ehrhart-theoretic counting problems, such as placements of nonattacking
chess pieces and Sidon sets. In particular, they obtain infinite depth for
the chromatic and characteristic polynomials of Kneser graphs, Johnson
graphs, and fixed Cartesian powers of complete graphs.

Section~\ref{sec:contrast} shows that chromatic and characteristic
polynomials can nevertheless exhibit different two-level behavior. For any
family of strongly regular graphs whose parameters are polynomial in $q$,
we prove that the characteristic polynomials have infinite depth over the
set of admissible indices (Theorem~\ref{thm:stronglyregular}). The proof
uses traces of powers of the adjacency matrix and Newton's identities.
In contrast, the chromatic polynomials of the family of Paley graphs have maximal depth two
(Theorem~\ref{thm:Paley}): their fourth-leading coefficient is not
eventually polynomial in $q$.

Section~3 connects two-level polynomials with strongly polynomial graph
sequences. We prove that the chromatic and characteristic polynomials of every
strongly polynomial graph sequence have infinite depth
(Theorem~\ref{thm:stronglypolynomial}), thus giving a common explanation
for the examples of Kneser graphs, Johnson graphs, and powers of complete graphs. 
The proof expresses each coefficient as a finite linear combination of counts of fixed subgraphs.

Finally, the identity $\chi(G_q,n)=\hom(G_q,K_n)$ leads us to consider
$\hom(G_q,H_n)$ when both graph sequences are strongly polynomial.
We give examples showing that infinite depth need not hold, and study
target families for which it holds for every strongly polynomial source
family. These include fixed categorical powers of complete graphs and an
ordered version of Kneser graphs. We conclude with questions about
their classification and the possible depth functions of homomorphism
counts.


  
\section{Characteristic polynomials vs. chromatic polynomials} \label{sec:contrast}
We recall the definitions of two important polynomials associated to a graph $G$. 

\begin{defn} Let $G$ be a graph.
\begin{enumerate}
\item The \emph{characteristic polynomial} $\psi_G(x)$ is the characteristic polynomial of the adjacency matrix $A(G)$. 
\item The \emph{chromatic polynomial} $\chi_G(x)$ is the unique polynomial such that for each natural number $n$, $\chi_G(n)$ equals the number of proper $n$-colorings of $G$.  
\end{enumerate}
\end{defn}

In this section we will prove an infinite-depth result for the characteristic polynomials of a certain type of graph, as follows. 

\begin{defn} \cite{GR}
A graph $G$ is \emph{strongly regular} with parameters $(v,k,\lambda,\mu)$ if it has $v$ vertices, is $k$-regular, every two adjacent vertices have exactly $\lambda$ common neighbors, and every two non-adjacent vertices have exactly $\mu$ common neighbors.  
\end{defn}
The fact that the characteristic polynomials of strongly regular graphs should be well-behaved is suggested by the following complete classification of their spectra.

\begin{prop} \cite[\S 10.2]{GR} \label{prop:spectrum}
The spectrum of a connected, noncomplete strongly regular graph with parameters
$\operatorname{SRG}(v,k,\lambda,\mu)$ consists of the eigenvalues
\[
k,\qquad
r=\frac{(\lambda-\mu)+\Delta}{2},
\qquad
s=\frac{(\lambda-\mu)-\Delta}{2},
\]
where
\[
\Delta=\sqrt{(\lambda-\mu)^2+4(k-\mu)}.
\]
The corresponding multiplicities of these eigenvalues are, respectively,
\[
1,\qquad
m_r=\frac{1}{2}
\left(
(v-1)
-
\frac{2k+(v-1)(\lambda-\mu)}{\Delta}
\right),
\qquad
m_s=\frac{1}{2}
\left(
(v-1)
+
\frac{2k+(v-1)(\lambda-\mu)}{\Delta}
\right).
\]
\end{prop}

There are relatively few known families of strongly regular graphs, but they do include Cartesian products of complete graphs $K_a \square K_a$ and complete bipartite graphs $K_{a,a}$. An additional example is given by a subclass of the following important class of graphs. 

\begin{defn}
For positive integers $a$ and $k$, the \emph{Kneser graph} $\Kn(a,k)$ is the graph on the $k$-subsets of $[a]$ for which two subsets $S, T$ are connected by an edge whenever $S \cap T = \emptyset$. 
\end{defn}

Kneser graphs of the form $\Kn(a,2)$ are strongly regular. However, for $k \geq 3$, the Kneser graph $\Kn(a,k)$ is not strongly regular because non-adjacent vertices are given by sets $S,T$ whose intersection can be of any size between $1$ and $k-1$, and the size of the intersection affects the number $N(S,T)$ of common neighbors of $S$ and $T$. For example, in $\Kn(a,3)$, we have 
\[ N(S,T) = \begin{cases} \binom{a-6}{3} & \text{if } S \cap T = \emptyset \\ \binom{a-5}{3} & \text{if } |S \cap T| = 1 \\ \binom{a-4}{3} & \text{if } |S \cap T| = 2 \end{cases}. \]

We will show that the characteristic polynomials of families of strongly regular graphs yield two-level polynomials of infinite depth. In particular, this reproves this result for Cartesian products of two complete graphs $K_q \times K_q$ and for the Kneser graphs $\Kn(q,2)$ in \cite{BW} in an entirely different way. However, we do not recover this result for $K_q^k$ nor for $\Kn(q,k)$ when $k \geq 3$. 

\begin{thm}\label{thm:stronglyregular}
Let $Q\subseteq\mathbb N$ be an infinite set, and let $\{G_q\}_{q\in Q}$ be a
family of strongly regular graphs with parameters
$v(q),k(q),\lambda(q),\mu(q)\in\mathbb Q[q]$.
Write
\[
\psi_{G_q}(x)=\sum_{r=0}^{v(q)}c_r(q)x^{v(q)-r},
\qquad c_r(q)=0\quad\text{for }r>v(q).
\]
For every $r\geq0$, $c_r(q)$ is a polynomial for all $q\in Q$.
Thus the characteristic polynomials form a two-level polynomial with degree
function $v$ and infinite depth over $Q$.
\end{thm}

\begin{proof}
Let $A_q$ be the adjacency matrix of $G_q$, and let $I_q$ and $J_q$ be
the identity and all-ones matrices of order $v(q)$, respectively.
For vertices $u,w$, the entry $(A_q^2)_{uw}$ counts the walks of
length two from $u$ to $w$. If $u=w$, this number is $k(q)$; if
$u\neq w$, it is the number of common neighbors of $u$ and $w$, namely
$\lambda(q)$ when they are adjacent and $\mu(q)$ otherwise. Hence
\begin{align*}
A_q^2&=k(q)I_q+\lambda(q)A_q+\mu(q)(J_q-I_q-A_q)\\
&=a_j(q)A_q +b_j(q)\left( k(q)I_q + \lambda(q) A_q + \mu(q)(J_q-I_q-A_q) \right) +d_j(q)k(q)J_q \\ 
&=(k(q)-\mu(q))I_q+(\lambda(q)-\mu(q))A_q+\mu(q)J_q.
\end{align*}
Also, $A_qJ_q=k(q)J_q$, since every row of $A_q$ has sum $k(q)$.

We claim that, for each fixed $j\geq0$, there exist polynomials
$a_j(x),b_j(x),d_j(x)\in\mathbb Q[x]$ such that
\[
A_q^j=a_j(q)I_q+b_j(q)A_q+d_j(q)J_q
\qquad(q\in Q).
\]
For $j=0$, take $(a_0,b_0,d_0)=(1,0,0)$. If the claim holds for $j$,
multiplication by $A_q$ gives
\begin{align*}
A_q^{j+1}
&=a_j(q)A_q+b_j(q)A_q^2+d_j(q)A_qJ_q\\
&=(k(q)-\mu(q))b_j(q)I_q+
\bigl(a_j(q)+(\lambda(q)-\mu(q))b_j(q)\bigr)A_q+
\bigl(\mu(q) b_j(q)+k(q)d_j(q)\bigr)J_q.
\end{align*}
The three coefficients are again polynomials in $q$, completing the
induction.
Since $\operatorname{tr}(A_q)=0$, it follows that, for every fixed $j\geq1$,
\[\operatorname{tr}(A_q^j)
=v(q)\bigl(a_j(q)+d_j(q)\bigr)
\]
is a polynomial in $\mathbb Q[q]$.

Now, Newton's identities for the roots of a polynomial in terms of its coefficients give
\[
r c_r(q)=-\sum_{j=1}^r c_{r-j}(q)\operatorname{tr}(A_q^j)
\qquad(r\geq1),
\]
with $c_0(q)=1$. 
Induction on $r$ therefore shows that $c_r(q)$ is polynomial in $q$ for all $r$.
\end{proof}

\subsection{Paley graphs}
\begin{defn}
  Let $q$ be a prime such that $q \equiv 1$ (mod 4). The \emph{Paley graph} $P_q$ is the graph on $\Z / q\Z$ where two vertices $x, y$ are adjacent if $x-y$ is a quadratic residue (mod $q$).
\end{defn}

The condition that $q \equiv 1$ (mod 4) guarantees that this edge relation is symmetric. Paley graphs are known to form a quasirandom family \cite[\S9]{AS}. Furthermore, they are strongly regular.

  \begin{prop} \label{prop:PaleySRG} \cite[\S 10.3]{GR}
    The graph $P_q$ is strongly regular with parameters $(q,\frac{q-1}{2},\frac{q-5}{4},\frac{q-1}{4})$. 
  \end{prop}

    By Theorem \ref{thm:stronglyregular}, it follows that the characteristic polynomials of the family of Paley graphs yield a two-level polynomial of infinite depth over the set of primes congruent to $1$ modulo $4$.



In contrast to the characteristic polynomial, we will show that the depth of the family of chromatic polynomials of Paley graphs $P_q$ is only two; that is, they do not exhibit two-level behavior. To prove this, we must show that while the three leading coefficients of $\chi_{P_q}(n)$ are given by polynomials in $q$, the coefficient of $n^{q-3}$ is not. Write
\begin{equation} \chi_{P_q}(n) = \sum_{r=0}^\infty c_r(q)n^{q-r}  \label{eq:Paley}  \end{equation}
where $c_r(q) = 0$ for $r > q$ and $c_r(q)$ is only defined when $q$ is a prime congruent to 1 (mod 4).  


\begin{prop}
The three leading coefficients in \ref{eq:Paley} are given by polynomials in $q$. 
\end{prop}
\begin{proof}
  The leading coefficient of the chromatic polynomial of any graph is one. The codegree one coefficient is $-m$, where $m$ is the number of edges. By Proposition \ref{prop:PaleySRG}, for $P_q$ we have $m =\frac{q(q-1)}{4}$, which is a polynomial. It is well known that the codegree two coefficient is $\binom{m}{2}-t$, where $t$ is the number of triangles. In the case of Paley graphs, we have by Proposition \ref{prop:PaleySRG} that $t = \frac{m\lambda}{3}= \frac{q(q-1)(q-5)}{48}$. 
\end{proof}

\begin{thm} \label{thm:Paley}
The sequence $\chi_{P,q}(n)$ is a two-level polynomial of degree function $q$ and maximal depth function two.  
\end{thm}

Farrell \cite[Theorem 1]{Farrell} showed that the codegree three coefficient of the chromatic polynomial of any graph $G$ with $m$ edges, $\Delta$ triangles, $t_1$ induced 4-cycles, and $t_2$ induced copies of $K_4$ is
\[ -\binom{m}{3} + (m-2)\Delta + t_1 - 2t_2.\]
To prove that the codegree three coefficient is not a polynomial in the case of Paley graphs, we combine this formula with the following result of Evans, Pulham and Sheehan.

\begin{thm}\cite{evans1981number} \label{thm:4cliques}
The number of 4-cliques in the Paley graph of order $q$ is 
\[ t_2(q)=\frac{q(q-1)\left((q-9)^2 - 4y(q)^2\right)}{2^9 \cdot 3},\]
where $y(q)$ is the unique positive even number $y$ such that $q=x^2+y^2$.
\end{thm}

\begin{lem} \label{lem:4cliques}
The expression for $t_2(q)$ given in Theorem \ref{thm:4cliques} is not a polynomial in $q$ for all sufficiently large $q$.
\end{lem}

\begin{proof} Suppose, for a contradiction, that $t_2(q)$ is eventually given by a polynomial. Solving for $y^2$, we obtain that $y^2$ is given by a rational expression in $q$. Write $y(q)^2=P(q)+\frac{R(q)}{G(q)}$, where $P$, $R$, and $G$ are polynomials and either $R$ is of strictly lower degree than $G$ or $R$ is the zero polynomial. In the former case, we would have $0<|\frac{R(q)}{G(q)}|<1$ for sufficiently large $q$ and $\lim_{q\to \infty}\frac{R(q)}{G(q)}=0$. Multiplying by a suitable integer $D$ such that $DP(q)\in \Z[q]$, we have $Dy(q)^2=DP(q)+D\frac{R(q)}{G(q)}$, so $D\frac{R(q)}{G(q)}$ must be an integer for all $q$, a contradiction. We conclude that $y(q)^2=P(q)$ is eventually a polynomial in $q$.

Since $q=x^2+y^2$, both $y(q)^2=P(q)$ and $x(q)^2=q-P(q)$ are eventually affine
functions of $q$. At least one has nonzero slope. Consequently, for
one of the integer-valued functions $z=x$ or $z=y$, there exist
integers $A,B,D$ with $A\neq0$ and $D>0$ such that
\[
z(q)^2=\frac{Aq+B}{D}
\]
for all sufficiently large $q\in Q$.
Choose a prime $\ell\geq5$ with $\ell\nmid AD$.
As $u$ ranges over $\mathbb F_\ell^{\times}$, the values
$D^{-1}(Au+B)$ are $\ell-1$ distinct elements of $\mathbb F_\ell$.
Only $(\ell+1)/2$ elements of $\mathbb F_\ell$ are squares, including
zero. Since $\ell-1>(\ell+1)/2$, some nonzero $u$ makes
$D^{-1}(Au+B)$ a quadratic nonresidue.
By the Chinese remainder theorem, the conditions
\[
q\equiv1\pmod4,\qquad q\equiv u\pmod\ell
\]
define a residue class coprime to $4\ell$. Dirichlet's theorem gives arbitrarily large primes
in this class. For any sufficiently large such prime, reducing the
identity for $z(q)^2$ modulo $\ell$ contradicts the choice of $u$.
\end{proof}

\begin{proof}[Proof of Theorem \ref{thm:Paley}] The codegree three coefficient is $-\binom{m}{3}+(m-2)\Delta+t_1-2t_2$, where $m:=m(q)$ is the number of edges, $\Delta:=\Delta(q)$ is the number of triangles, $t_1:=t_1(q)$ is the number of induced 4-cycles, and $t_2$ is the number of cliques of size $4$ as before. Because Paley graphs are strongly regular, $m(q)$ and $\Delta(q)$ are always polynomials, so it will suffice to show that $t_1-2t_2$ is not a polynomial. Let $t_3$ be the number of induced subgraphs which are isomorphic to $K_4$ minus a single edge.

First, notice that $2t_1+t_3$ is a polynomial. Indeed, this quantity counts the number of ways to choose two non-adjacent vertices together with two common neighbors, which is a polynomial because the graph is strongly regular. If the two chosen neighbors are non-adjacent, then the induced graph is a 4-cycle, and this graph is counted twice (according to which pair of opposite vertices is chosen). In the other case the induced graph is a $K_4$ minus an edge, counted only once. Similarly, $6t_2+t_3$ is a polynomial. In this case, we choose two adjacent vertices and then two common neighbors. A copy of $K_4$ contributes $6$ such choices, while a copy of $K_4$ minus an edge contributes exactly one. 

Therefore,
$(2t_1+t_3)-(6t_2+t_3)=2t_1-6t_2$
is a polynomial. Suppose, by contradiction, that $t_1-2t_2$ is also a polynomial $g(q)$. Then by solving the nondegenerate linear system 
\[ 2t_1 - 6t_2 = f(q), \, t_1-2t_2 = g(q)\]
we obtain that $t_2$ is also a polynomial. But this contradicts Lemma \ref{lem:4cliques}, completing the proof.
\end{proof}

\section{Strongly polynomial families and two-level behavior}
The work of Bogart and Woods contains many examples of families of graphs whose chromatic and characteristic polynomials are two-level polynomials with infinite depth. The examples they develop are Kneser graphs $\Kn(q,k)$, Johnson graphs $\textup{Jn}(q,k)$, and Cartesian products of complete graphs $K_q^k$ for $k$ fixed. It is natural to ask if it is possible to systematically construct a wider variety of graph families that also have this property. In this section we show that this property holds for any strongly polynomial sequence of graphs, a notion first defined in \cite{Jaeger}. Moreover, all the particular examples in Bogart and Woods are indeed families of this kind.

In this section, we denote by $\textbf{Hom}(F,G)$ the set of homomorphisms between two graphs $F=(V,E)$, $G=(V',E')$; that is, the set of functions $\zeta:F\rightarrow G$ such that every $(v,w)\in E$, we have $(\zeta(v),\zeta(w))\in E'$. We denote the cardinality of $\textbf{Hom}(F,G)$ by $\hom(F,G)$. Similarly, we denote by $\textbf{Inj}(F,G)$ the set of injective homomorphisms from $F$ to $G$ and by  $\text{inj}(F,G)$ the cardinality of this set. These two quantities are related via M\"obius inversion. 

\begin{prop}
\label{prop:inj and hom}
\cite[Lemma 2.3]{GNO} \label{prop:injhomcount}
Let $F$ and $G$ be graphs. Then
$$\text{inj}(F,G)=\sum\limits_{\theta}\mu(\theta)\hom(F/\theta,G)$$
where $\mu$ is the M\"obius function of the partition lattice of the vertices of $F$.
\end{prop}
\begin{proof}
Let $h\in \textbf{Hom}(F,G)$. Then $h$ naturally induces a partition $\theta$ of the vertices of $F$, where $v,w$ belong to the same block if $h(v)=h(w)$. Let $F/\theta$ be the graph obtained by identifying the vertices in each block of $\theta$ into a single vertex, which inherits all the edges from the vertices in the block. Then $h$ can be regarded as an injective homomorphism from $F/\theta$ to $G$. This shows that
\begin{equation}\hom(F,G)=\sum\limits_{\theta \in \Pi(V(F))}\text{inj}(F/\theta, G) \label{eq:homtoinj}\end{equation}
where $\Pi(V(F))$ is the lattice of partitions of the vertex set of $F$. Note that even if $\theta$ contains a block which is not an independent set, the number $\text{inj}(F/\theta, G)$ will be zero. By M\"obius inversion in this lattice,
\begin{equation}\text{inj}(F,G)=\sum\limits_{\theta}\mu(\theta)\hom(F/\theta,G). \qedhere \label{eq:injtohom} \end{equation} 
\end{proof}

\begin{defn}
A sequence $\{G_n\}_{n\geq 0}$ of graphs is \textit{strongly polynomial} if, for every fixed graph $F$, there exists a polynomial $p_F(n)$ such that $\hom(F,G_n)=p_F(n)$. If this polynomial exists only for $n\geq n_F$ for a constant $n_F$ (which depends on $F$), we say that the sequence is \textit{polynomial}.
\end{defn}

A simple example of a graph sequence which is polynomial but not strongly polynomial is provided by the path graph $L_n$.

\begin{prop}\label{prop:paths-polynomial}
Let $L_n$ be the path on $n$ vertices, where $n\geq1$.
Then $\{L_n\}_{n\geq1}$ is a polynomial sequence of graphs but is not strongly polynomial.
\end{prop}

\begin{proof}
Let $F$ be a fixed nonempty connected graph with $v$ vertices, and let
$a_\ell$ be the number of vertex-surjective homomorphisms from $F$
to $L_\ell$. The image of any homomorphism from $F$ to $L_n$ is a
connected set of at most $v$ vertices, and hence an interval of the
path. There are $n-\ell+1$ intervals with $\ell$ vertices when
$\ell\leq n$, and none otherwise. For each such interval, the number
of homomorphisms with exactly that vertex image is $a_\ell$.
Therefore
\[
\hom(F,L_n)=\sum_{\ell=1}^{v}a_\ell\max\{n-\ell+1,0\}.
\]
For $n\geq v$, this agrees with the polynomial
\[
\sum_{\ell=1}^{v}a_\ell(n-\ell+1).
\]
If $F$ has connected components $F_1,\ldots,F_c$, then
\[
\hom(F,L_n)=\prod_{i=1}^{c}\hom(F_i,L_n),
\]
which is again eventually polynomial in $n$. The empty graph has
exactly one homomorphism to every $L_n$, so this also covers that case.

To see that $L_n$ is not strongly polynomial, observe that
$$\operatorname{hom}(L_3,L_n)=
\begin{cases}
0 & n=1,\\
4n-6 & n\geq 2.
\end{cases}$$
Indeed, for \(n\geq2\), the image of the central vertex of \(L_3\) can be chosen in \(n\) ways, and the remaining choices give a unique homomorphism when the image is an endpoint and \(4\) when it is an interior vertex. Thus, $
\operatorname{hom}(L_3,L_n)=2+4(n-2)=4n-6$,
which is only valid for \(n\geq2\). Hence, $L_n$ is only a polynomial sequence, but not a strongly polynomial one.
\end{proof}

The property of being a strongly polynomial sequence is preserved by many important operations in graph theory, including the most common types of graph products(cartesian, categorical, strong, and lexicographical), disjoint unions, joins, and complements \cite{Jaeger}. A motivating example of this concept is the sequence of complete graphs. For any graph $F$, $\hom(F,K_n)$ is the number of $n$-colorations of $F$, so $p_F(n)$ is the chromatic polynomial of $F$. Other strongly polynomial sequences include the Kneser graphs $\text{Kn}(n,k)$ with $k$ fixed, the Johnson graphs $\text{Jn}(n,k)$ with $k$ fixed, and the \emph{crown graphs} $\text{Cr}(n)$ which are obtained by removing a single perfect matching from the complete bipartite graphs $K_{n,n}$.  

 The authors of \cite{GNO} develop a general method to construct all the strongly polynomial graph families that are found in the literature, using concepts from model theory such as interpretations of relational structures. However, it is an open problem whether all strongly polynomial families can be constructed in this way.

\begin{prop}
\label{prop:injpol}
Let $\{G_n\}$ be a strongly polynomial family of graphs. Then for any fixed graph $F$, $\text{inj}(F,G_n)$ is a polynomial in $n$. The reverse implication is also true.
\end{prop}
\begin{proof}
By Proposition \ref{prop:inj and hom}, $\text{inj}(F,G_n)$ can be expressed as a finite linear combination of terms of the form $\hom(F/\theta,G_n)$ for some partition $\theta$. By hypothesis, each of these terms is a polynomial, so $\text{inj}(F,G_n)$ is also a polynomial. The other direction follows from (\ref{eq:homtoinj}).
\end{proof}

\begin{thm} \label{thm:stronglypolynomial} Let $G_n$ be a strongly polynomial family of graphs. Then $\chi(G_n,x)$ and $\psi(G_n,x)$ are two-level polynomials of infinite depth.
\end{thm}

\begin{proof} To prove that $\chi(G_n,x)$ is a two-level polynomial, we use the following well-known expansion in terms of subsets of edges \cite[Theorem 10.4]{Biggs}:
\[ \chi(G_n,x)=\sum_{S\subset E(G_n)}(-1)^{|S|}x^{v(n)-R(S)} \]
where $v(n)$ is the number of vertices of $G_n$ and, for a subset $S$ of edges, $R(S)$ is the number of vertices of the subgraph induced by $S$ minus the number of connected components. Because $\hom(K_1,G_n)$ is a polynomial, it follows that $v(n)$ is a polynomial in $n$. Now, for a fixed codegree $r$, there are only finitely many subsets of edges $S$ such that $R(S)=r$, up to isomorphism. Let $H_1,...H_m$ be these isomorphism classes of graphs induced by such subsets of edges. We want to count the number of copies, not necessarily induced, of these graphs in $G_n$. This number is exactly $\text{inj}(H_i,G_n)$ divided by the size of the automorphism group of $H_i$. Hence, the coefficient of the fixed codegree $r$ is 
$$\varphi_r(n)=\sum_{i=1}^{m}(-1)^{E(H_i)}\frac{\text{inj}(H_i,G_n)}{|\text{Aut}(H_i)|}$$
By Proposition \ref{prop:injpol}, we know that every term $\text{inj}(H_i,G_n)$ is a polynomial, so $\varphi_r(n)$ is also a polynomial, showing that $\chi(G_n,x)$  is a two-level polynomial of infinite depth.

To prove the same property for $\psi(G_n,x)$, we use a similar expansion. By \cite[Proposition 7.3]{Biggs}, 
\[\psi(G_n,x)=\sum\limits_{r=0}^{\infty}\sum\limits_{v(\Gamma)=r}(-1)^{r(\Gamma)}(-2)^{s(\Gamma)} x^{v(n)-r}\]
where all the subgraphs $\Gamma$ have connected components which are single edges or cycles, and $r(\Gamma)$ and $s(\Gamma)$ represent the numbers of connected components of the first and second types, respectively. For a fixed $r$, there are only finitely many isomorphism classes of graphs of this type, so, as in the chromatic case, the number of copies, not necessarily induced, of these graphs in $G_n$ is a polynomial, and this proves that the characteristic polynomial is also a two-level polynomial with infinite depth.
\end{proof}

\begin{rem} A sequence of graphs can have a chromatic polynomial which is a two-level polynomial with infinite depth without being strongly polynomial. A simple example is given by the path graph $L_q$, whose chromatic polynomial is $n(n-1)^{q-1}$, just as for all trees on $q$ vertices. By Proposition \ref{prop:paths-polynomial}, the sequence $\{L_q\}$ is not strongly polynomial. 

The path graph also provides an example of a graph sequence whose chromatic polynomial is a two-level polynomial of infinite depth, although its characteristic polynomial has only finite depth. Let $\psi(L_q,n)$ be this characteristic polynomial. Because the spectrum of this graph is $\{2\cos(\dfrac{k\pi}{q+1})$, $1\leq k\leq q\}$, it follows that $\psi_{L_q}(n)=U_q(n/2)$, where $U_q$ is the $q$-th Chebyshev polynomial of the second kind \cite[p.~11]{Biggs}. Thus
\[ \psi_{L_q}(n)=\sum\limits_{r=0}^{\lfloor q/2 \rfloor}(-1)^{r}\binom{q-r}{r}n^{q-2r} \]

This is a two-level polynomial of finite depth, because $\varphi_r(q)=0$ when $r$ is odd, and if $r$ is even, $\varphi_{r}(q)=(-1)^{r/2}\binom{q-r/2}{r/2}$ is $0$ when $q<r$, and agrees with the polynomial \[(-1)^{r/2}\dfrac{(q-\frac{r}{2})(q-\frac{r}{2}-1)\cdots(q-r+1)}{\frac{r}{2}!}\] only for $q\geq r/2$. Hence the depth of $\psi_{L_q}(n)$ is $2q$.
\end{rem}

Since the proper $n$-colorings of a graph $G$ are exactly the homomorphisms from $G$ to the complete graph $K_n$, Theorem \ref{thm:stronglypolynomial} states that the number of homomorphisms from \emph{any} strongly polynomial sequence $\{G_q\}$ to the \emph{particular} strongly polynomial sequence $\{K_n\}$ is given by a two-level polynomial of infinite depth. This suggests several questions. Most optimistically, could it be the case that for \emph{any pair} of strongly polynomial sequences $\{G_q\}$ and $\{H_n\}$, the number of homomorphisms is given by a two-level polynomial of infinite depth? The answer turns out to be no. 

\begin{prop} \label{prop:counterexamples}
  \begin{enumerate}
  \item Let $S_n$ be the star graph on $n+1$ vertices (that is, a central vertex surrounded by $n$ leaves). Then $\hom(S_q,S_n)=n^q+n$, which is a two-level polynomial of depth $q-2$ as in Example \ref{ex:finitedepth}.
  \item Furthermore, $\hom(K_q,K_{n^2} \sqcup K_n)=(n^2)_q +(n)_q$ for $q\geq 1$, which is a two-level polynomial of depth $q-1$. 
  \end{enumerate}
\end{prop}
  \begin{proof}
  \begin{enumerate}
      \item[(1)] Let $s$ and $s'$ be the central vertices of $S_q$ and $S_n$, respectively. A graph homomorphism $f$ from $S_q$ to $S_n$ may send $s$ to $s'$, and in this case each of the $q$ leaves of $S_q$ must be sent to one of the $n$ leaves of $S_n$, which gives $n^q$ options. If $s$ is sent to any of the $n$ leaves, then the other leaves must be sent to $s'$, giving another $n$ options.

      \item[(2)] Any $q$-clique in $K_{n^2} \sqcup K_n$ must be contained in $K_{n^2}$ or in $K_n$, so $\hom(K_q,K_{n^2} \sqcup K_n)=\hom(K_q,K_{n^2})+\hom(K_q,K_n)=(n^2)_q +(n)_q$. For each $q$ this is a polynomial in $n$ of degree $2q$ and the terms agree with the infinite depth two-level polynomial $(n^2)_q$ down to degree $q+1$ but no further. Thus the depth of this sequence of polynomials is $q-1$.  
  \end{enumerate}
  \end{proof}
  
  By Theorem \ref{thm:stronglypolynomial}, we have that $\hom(G_q, K_n)$ is a two-level polynomial of infinite depth. Proposition \ref{prop:counterexamples} (2) implies that this conclusion does not continue to hold if we exchange the source (i.e., a general strongly polynomial family) with the target (i.e., the family of complete graphs). 

  However, we mention one case for which homomorphisms \emph{from} a particular strongly polynomial family \emph{to} any strongly polynomial family do yield infinite depth, at least up to global multiplication by an appropriate function of $q$.  

\begin{prop}
  Let $M_q$ be a matching of size $q$ (that is, the union of $q$ disjoint edges), and let $G_n$ be any strongly polynomial family of graphs. Then $\hom(M_q,G_n)$ is a two-level polynomial with infinite depth multiplied by $c^q$, where $c$ is a fixed constant.
\end{prop}

\begin{proof}
   Every edge of a matching must be sent to an edge, and the order of the two vertices in this edge matters, so $\hom(M_q,G_n)=(2|E(G_n)|)^q=2^q|E(G_n)|^q$, where $|E(G_n)|$ is the number of edges of $G_n$, which is a polynomial in $n$ because $G_n$ is strongly polynomial. If the leading coefficient of $|E(G_n)|$ is $k$, so that $|E(G_n)|=kf(n)$ for a monic polynomial $f(n)$, then $\hom(M_q,G_n)=(2k)^qf(n)^q$ and, by \cite[Proposition 18]{BW}, $f(n)^q$ is a two-level polynomial of infinite depth.
\end{proof}

We now return to the situation that the source is an arbitrary strongly polynomial family, and consider possible targets beyond the family of complete graphs. 

\begin{defn}
Let $H_n$ be a family of graphs. We say $H_n$ is an (infinite-depth) \emph{two-level target sequence} if, for every strongly polynomial family of graphs $G_q$, we have that $\hom(G_q,H_n)$ is a two-level polynomial of infinite depth.
\end{defn}

By taking a fixed graph in the domain, it is easy to see that any two-level target sequence must be strongly polynomial. Theorem \ref{thm:stronglypolynomial} shows that $K_n$ is a two-level target sequence. The next lemma establishes that this property is preserved under categorical products of graphs.

\begin{lem}
 Let $H_n$ and $J_n$ be two two-level target sequences. Then $H_n\times J_n$ is also a two-level target sequence, where the product is the categorical product.
\end{lem}
\begin{proof}
If $G_q$ is a strongly polynomial sequence of graphs, then 
\[\hom(G_q,H_n\times J_n)=\hom(G_q,H_n)\hom(G_q,J_n)\] 
and this product is a two-level target sequence by Proposition \ref{prop:closure}.
\end{proof}

\begin{cor}
 Let $K_n^k$ be the categorical product of $k$ copies of $K_n$ (that is, two $k$-tuples are adjacent if they differ in every coordinate). Then $K_n^k$ is a two-level target sequence.
\end{cor}


Going beyond categorial products of complete graphs, we next consider Kneser graphs. First, notice that $\hom(K_q,\Kn(n,k))$, that is, the number of $q$-cliques in $\Kn(n,k)$ multiplied by $q!$, is exactly $\binom{n}{k}\binom{n-k}{k}\cdots \binom{n-(q-1)k}{k}$, which is a two-level polynomial after factoring out the term $(k!)^q$ in the denominator. To avoid this term, we define an auxiliary graph $\overline{\Kn}(n,k)$ to be the graph whose vertices are ordered $k$-tuples of distinct elements of $[n]$, and whose edges are given by pairs of tuples whose underlying sets are disjoint. Equivalently, we obtain $\overline{\Kn}(n,k)$ by starting with $\Kn(n,k)$ and copying each vertex $k!$ times. 

\begin{thm} \label{thm:Knesertarget}
For $k$ fixed, the sequence $\{\overline{\Kn}(n,k)\}$ is a two-level target sequence. 
\end{thm}

\begin{lem}
If $F$ is a graph with $q$ vertices, then $\hom(F,Kn(n,k))=(k!)^q\hom(F,\overline{Kn}(n,k))$.
\end{lem}
\begin{proof} Every homomorphism $\Phi$ from $F$ to $\Kn(n,k)$ can be extended to $\overline{Kn(n,k)}$ by selecting, for every vertex $v$ of $F$, any of the $k!$ permutations of the image $\Phi(v)$.
\end{proof} 

\begin{lem}\label{lem:toKneser}
 Let $G=(V,E)$ be a graph, and let $G \cdot K_k$ be the lexicographic product with the complete graph $K_k$: that is, the vertex set is $F\times [k]$ and two vertices $(v,i),(w,j)$ are connected by a edge if $v=w$ and $i\neq j$, or $(v,w)\in E$. Then
\[ \hom(G,\overline{Kn}(n,k))=\hom(G\cdot K_k,K_n)=\chi(G\cdot K_k,n).\]
\end{lem}

\begin{proof} Let $\psi\in \textbf{Hom}(G,\overline{Kn}(n,k))$. and define a function $c_{\psi}:G\cdot K_k\rightarrow K_n$ by $c_{\psi}(g,i)=\psi_i(g)$, where $\psi_i$ is the $i$-th coordinate of the function $\psi$. Then $c_{\psi}$ is a valid graph homomorphism, since $c_{\psi}(g,i)$ and $c_{\psi}(g,j)$ are different when $i\neq j$ and $c_{\psi}(v,i)$, $c_{\psi}(w,j)$ are also different when $(v,w)\in E$. 

For the other direction, if $c$ is a proper coloring of $G\cdot K_k$, then $\hat{c}:G\rightarrow \overline{Kn}(n,k)$ defined by $\hat{c}(v)=(c(v,1),c(v,2),...c(v,k))$ is also a valid graph homomorphism, since $c(v,i)\neq c(v,j)$ when $i\neq j$ and if $(v,w)\in E$, then all the coordinates in $\hat{c}(v)$ and $\hat{c}(w)$ are distinct.
\end{proof}

\begin{proof}[Proof of Theorem \ref{thm:Knesertarget}]
It follows from \cite[Corollary 4.6]{GNO} that $G_n\cdot K_k$ is itself a strongly polynomial sequence of graphs. Then, if $G_q$ is a strongly polynomial sequence of graphs, $\hom(G_q,\overline{Kn}(n,k))$ is equal to $\hom(G_q\cdot K_k,K_n)$ by Lemma \ref{lem:toKneser}, which is a two-level polynomial of infinite depth by Theorem \ref{thm:stronglypolynomial}.
\end{proof}

\begin{prop}
If $H_n$ is a two-level target sequence, then for every fixed graph $F$, $\hom(F,H_n)$ must be a \textbf{monic} polynomial or $0$.
\end{prop}
\begin{proof}
Let $F_q$ be the graph sequence defined by \[F_q=\bigsqcup_{i=1}^{q}F\]
\vspace{-7mm}

By Lemma 2.12 of \cite{GNO}, $F_q$ is a strongly polynomial sequence. Then $\hom(F_q,H_n)=\hom(F,H_n)^q$ is a two-level polynomial by hypothesis. If the leading coefficient of $\hom(F,H_n)$ is $c$, then the leading coefficient of $\hom(F,H_n)^q$ is $c^q$, which is a polynomial only if $c$ is $0$ or $1$. The result follows.
\end{proof}
We end with some questions about the two-level behavior of homomorphisms between strongly polynomial families.

\begin{quest} Can we classify all possible two-level target sequences?
\end{quest}

\begin{quest} More generally, is it possible to find necessary and sufficient conditions on a pair of graph sequences $G_q$, $H_n$ for $\hom(G_q,H_n)$ to be a two-level polynomial of infinite depth?
\end{quest}

The two examples in Proposition \ref{prop:counterexamples}, though they show that the depth is not always infinite, still exhibit nontrivial two-level behavior in the following sense. In each family, the degree and depth functions are both polynomials of the same degree.  

\begin{quest}
Do there exist strongly polynomial families $G_q$, $H_n$ such that the depth function of the sequence $\hom(G_q, H_n)$ is a polynomial of strictly lower degree than the degree function? 
\end{quest}

\begin{quest}
More dramatically, do there exist strongly polynomial families $G_q$, $H_n$ such that the depth function of $\hom(G_q, H_n)$ is \emph{constant} while the degree function is a polynomial of positive degree? 
\end{quest}

\section*{Acknowledgements} Both authors were supported by the Faculty of Sciences of the Universidad de los Andes: Tristram Bogart by internal research grant INV-2025-213-3438, and Santiago Jim\'enez Salazar by a summer graduate research assistantship. The authors thank John Goodrick and Kevin Woods for useful discussions about strongly polynomial families and two-level polynomials. 

\bibliographystyle{amsplain}
\bibliography{references}
\end{document}